\documentclass[12pt,leqnod]{amsart}
\usepackage[utf8]{inputenc}
\usepackage[centertags]{amsmath}
\usepackage{newlfont}
\usepackage{amsmath}
\usepackage{amsthm}
\usepackage{amsfonts}
\usepackage{amssymb}
\usepackage{graphicx}
\usepackage{cancel}

\newtheorem{theorem}{Theorem}[section]
\newtheorem{definition}[theorem]{Definition}

\newtheorem{lemma}[theorem]{Lemma}
\newtheorem{proposition}[theorem]{Proposition}
\newtheorem{remark}[theorem]{Remark}

\numberwithin{equation}{subsubsection}
\usepackage[top=2.5cm,bottom=2.5cm,left=2.5cm,right=2.4cm]{geometry}
\title[A new approach to the S-functional calculus]{ A new approach to the S-functional calculus}
\author{El Hassan Benabdi and Mohamed Barraa}
\address{Department of Mathematics, Faculty of Sciences Semlalia, Cadi Ayyad University, Marrakesh, Morocco} 
\email{elhassan.benabdi@gmail.com}
\email{barraa@hotmail.com}
\date{}
\begin{document}
\maketitle
\begin{abstract}
In this paper, we first prove that the S-spectrum of a bounded right quaternionic linear operator on a two-sided quaternionic Banach space is a union of the spectrum of some bounded linear operators on a complex Banach space. Furthermore, we show that the S-functional calculus is obtained by the Riesz-Dunford functional calculus for complex linear operators. We also give simple proofs of some already existing results.
\end{abstract}

\section{Introduction}
Let $X$ be a complex Banach space, we denote by $\mathcal{B}(X)$ the set of all bounded linear operators on $X$. For $T\in\mathcal{B}(X)$, $\rho(T)$, $\sigma(T)$ and $r(T)$ shall denote its resolvent set, its spectrum set and its spectral radius, respectively. Let $U$ be an open set in the plane and let $\mathcal{H}(U)$ be the set of all complex valued functions defined and holomorphic on $U$. Suppose that $D$ is a bounded Cauchy domain such that $\sigma(T)\subset D\subset\overline{D}\subset U$, the Riesz-Dunford functional calculus for $T$ is defined for any function $f\in \mathcal{H}(U)$ by
$$f(T):=\frac{1}{2\pi \mathrm{i}}\int_{\partial D}f(\lambda)(\lambda I-T)^{-1}d\lambda,$$
where $\partial D$ is the (oriented) boundary of $D$. Let $f,g\in \mathcal{H}(U)$, it is well-known that
\begin{itemize}
\item[•] $\sigma(f(T))=f(\sigma(T))$;
\item[•] $(fg)(T)=f(T)g(T)=g(T)f(T)$;
\item[•] let $V$ be an open set containing $f(U)$ and $h\in \mathcal{H}(V)$, then $(h \circ f)(T)=h(f(T))$;
\item[•] if $(f_n)$ converges to $f$ uniformly on compact subsets of $U$, then $f(T)=\lim_{n\rightarrow\infty}f_n(T)$. 
\end{itemize}

The aim of this paper is to show that almost every results of quaternionic spectral theory can be obtained by the classical spectral theory.\\

The article is organized as follows. In Section 2 we introduce the set of quaternions, slice hyperholomorphic functions, two-sided quaternionic Banach space and S-spectrum as needed for the development of this article. In Section 3 we give an equivalent definition of the S-spectrum and we give alternative proof of some already existing results. In Section 4 we show that the S-functional calculus can be obtained by the Riesz-Dunford functional calculus for complex linear operators. We also show that most of the properties that hold for the Riesz-Dunford functional calculus can be extended to the S-functional calculus.
\section{Preliminary results}
We denote by $\mathbb{H}$ the algebra of quaternions, introduced by Hamilton in 1843. An element $q$ of $\mathbb{H}$ is of the form
 $$q = a + b\mathrm{i} + c\mathrm{j} + d\mathrm{k};\; a,b,c,d\in\mathbb{R}$$
where $\mathrm{i},\mathrm{j}$ and $\mathrm{k}$ are imaginary units. By definition, they satisfy
$$\mathrm{i}^2=\mathrm{j}^2=\mathrm{k}^2=\mathrm{i}\mathrm{j}\mathrm{k}=-1.$$
Given $q=a+b\mathrm{i}+c\mathrm{j}+d\mathrm{k}$, then the conjugate quaternion, the norm, the real and the imaginary parts of $q$ are respectively defined by $\bar{q}=a-b\mathrm{i}-c\mathrm{j}-d\mathrm{k}$, $\vert q\vert=\sqrt{q\bar{q}}=\sqrt{a^2+b^2+c^2+d^2}$, $\text{Re}(q):=\frac{1}{2}(q+\bar{q})=a$ and $\text{Im}(q):=\frac{1}{2}(q-\bar{q})=b\mathrm{i} + c\mathrm{j} + d\mathrm{k}$.\\

The unit sphere of imaginary quaternions is given by
$$\mathbb{S}=\{q\in\mathbb{H}:q^2=-1\}.$$
Let $p$ and $q$ be two quaternions. $p$ and $q$  are said to be conjugated, if there is $s\in\mathbb{H}\setminus\{0\}$ such that $p=sqs^{-1}$. The set of all quaternions conjugated with $q$, is equal to the 2-sphere 
$$[q] =\{\text{Re}(q)+\vert\text{Im}(q)\vert j : j \in \mathbb{S}\}=\text{Re}(q)+\vert\text{Im}(q)\vert\mathbb{S}.$$
For every $j \in \mathbb{S}$, denote by $\mathbb{C}_j$ the real subalgebra of $\mathbb{H}$ generated by $j$; that is, 
$$\mathbb{C}_j:=\{\alpha+\beta j\in\mathbb{H}: \alpha,\beta\in\mathbb{R}\}.$$
We say that $U\subseteq\mathbb{H}$ is axially symmetric if $[q] \subset U$ for every $q \in U$.\\

For a thorough treatment of the algebra of quaternions $\mathbb{H}$, the reader is referred, for instance, to \cite{5ref10}.
\begin{definition}[{\cite[Definition 2.1.2]{5ref7}}, Slice hyperholomorphic functions] 
Let $U\subseteq\mathbb{H}$ be an axially symmetric open set and let $\mathcal{U}=\{(\alpha,\beta)\in\mathbb{R}^2: \alpha+\beta\mathbb{S}\subset U\}$.\\
A function $f : U\rightarrow\mathbb{H}$ is called a left slice function if there exist two functions $f_0, f_1 : \mathcal{U}\rightarrow\mathbb{H}$ such that:
$$f(q) = f_0(\alpha,\beta) + if_1(\alpha,\beta)\;\;\; \text{ for every } q = \alpha+\beta i \in U$$
and if $f_0, f_1$ satisfy the compatibility conditions
\begin{equation}
f_0(\alpha,-\beta) = f_0(\alpha,\beta),\;\;\;\;\;\; f_1(\alpha,-\beta) = -f_1(\alpha,\beta).\label{5eq6}
\end{equation}
If in addition $f_0$ and $f_1$ satisfy the Cauchy-Riemann equations
\begin{equation}
\begin{split}
\frac{\partial}{\partial \alpha}f_0(\alpha,\beta)-\frac{\partial}{\partial \beta}f_1(\alpha,\beta)= 0,\\
\frac{\partial}{\partial \beta}f_0(\alpha,\beta)+\frac{\partial}{\partial \alpha}f_1(\alpha,\beta)= 0,\label{5eq7}
\end{split}
\end{equation}
then $f$ is called left slice hyperholomorphic. We denote the set of all left slice hyperholomorphic functions on $U$ by $\mathcal{SH}_L(U)$.\\
A function $f : U\rightarrow\mathbb{H}$ is called a right slice function if there exist two functions $f_0, f_1 : \mathcal{U}\rightarrow\mathbb{H}$ such that:
$$f(q) = f_0(\alpha,\beta) + f_1(\alpha,\beta)i\;\;\; \text{ for every } q = \alpha+\beta i \in U$$
and if $f_0, f_1$ satisfy the compatibility conditions $(\ref{5eq6})$. If in addition $f_0$ and $f_1$ satisfy
the Cauchy-Riemann equations $(\ref{5eq7})$, then $f$ is called right slice hyperholomorphic. We denote the set of all right slice hyperholomorphic functions on $U$ by $\mathcal{SH}_R(U)$.\\
If $f$ is a left (or right) slice hyperholomorphic function such that $f_0$ and $f_1$ are real-valued, then $f$ is called intrinsic slice hyperholomorphic function. The set of all intrinsic slice hyperholomorphic functions on $U$ will be denoted by $\mathcal{N}(U)$.
\end{definition}
\begin{lemma}[{\cite[Lemma 2.1.5]{5ref7}}]
\label{5l6}
Let $U\subseteq\mathbb{H}$ be axially symmetric and let $f$ be a left (or right) slice function on $U$. The following statements are equivalent.
\begin{itemize}
\item[(i)] The function $f$ is intrinsic.
\item[(ii)] We have $f(U\cap\mathbb{C}_i) \subset\mathbb{C}_i$ for all $i\in\mathbb{S}$.
\item[(iii)] We have $f(\bar{q}) = \overline{f(q)}$ for all $q\in U$.
\end{itemize}
\end{lemma}
\begin{remark}
\label{5r1}
\emph{Let $U\subseteq\mathbb{H}$ be an axially symmetric open set and let $f\in\mathcal{N}(U)$. If we identify $\mathbb{C}_i$ with $\mathbb{C}$, by Lemma \ref{5l6} (ii) and the Cauchy-Riemann equations (\ref{5eq7}), $f\vert_{U\cap\mathbb{C}}$ is holomorphic on $U\cap\mathbb{C}$.}
\end{remark}
\begin{definition}
Let $(X,+)$ be an abelian group.
\begin{itemize}
\item $X$ is a right quaternionic vector space denoted by $X_R$ if it is endowed with a right quaternionic multiplication $(X,\mathbb{H})\rightarrow X$, $(x, q) \mapsto xq$ such that for all $x,y \in X$ and all $p, q \in\mathbb{H}$,
$$x(p + q) = xp + xq,\;\; (x+y)q = xq + yq,\;\;  (xp)q = x(pq)\; \text{ and }\; x1 = x.$$
\item $X$ is a left quaternionic vector space denoted by $X_L$ if it is endowed with a left quaternionic multiplication $(\mathbb{H},X)\rightarrow X$, $(q,x) \mapsto qx$ such that for all $x,y \in X$ and all $p, q \in\mathbb{H}$,
$$(p + q)x = px + qx,\;\;  q(x+y) = qx + qy,\;\;  q(px) = (qp)x\; \text{ and }\; 1x = x.$$
\item $X$ is a two-sided quaternionic vector space if it is endowed with a left and a right quaternionic  multiplication such that $X$ is both a left and a right quaternionic vector space and such that $rx = xr$ for all $r\in\mathbb{R}$, and $(px)q=p(xq)$ for all $p,q\in\mathbb{H}$ and all $x\in X$.
\end{itemize}
\end{definition}
\begin{definition}
Let $X_R$ be a right quaternionic vector space. A function  $\Vert\cdot\Vert : X_R \rightarrow [0;+\infty)$ is called a norm on $X_R$, if it satisfies
\begin{itemize}
\item[(i)] $\Vert x\Vert=0$ if and only if $x=0$;
\item[(ii)] $\Vert xq\Vert=\Vert x\Vert \vert q\vert$ for all $x\in X_R$ and all $q \in\mathbb{H}$;
\item[(iii)] $\Vert x+y\Vert\leq \Vert x\Vert+\Vert y\Vert$ for all $x,y\in X_R$.
\end{itemize}
If $X_R$ is complete with respect to the metric induced by $\Vert\cdot\Vert$, we call $X_R$ a right quaternionic Banach space.\\
Let $X_L$ be a left quaternionic vector space. A function $\Vert\cdot\Vert : X_L \rightarrow [0;+\infty)$ is called a norm on $X_L$, if it satisfies $(\mathrm{i})$, $(\mathrm{iii})$ and 
\begin{itemize}
\item[(ii')] $\Vert qx\Vert=\vert q\vert\Vert x\Vert$ for all $x\in X_L$ and all $q \in\mathbb{H}$.
\end{itemize}
If $X_L$ is complete with respect to the metric induced by $\Vert\cdot\Vert$, we call $X_L$ a left quaternionic Banach space.\\
Finally, a two-sided quaternionic vector space $X$ is called a two-sided quaternionic Banach space if it is endowed with a norm $\Vert\cdot\Vert$ such that it is both left and right quaternionic Banach space.
\end{definition}
\begin{remark}
\emph{If $X$ is a two-sided quaternionic Banach space, then  $\Vert qx\Vert=\Vert xq\Vert=\vert q\vert\Vert x\Vert$ for all $x\in X$ and all $q \in\mathbb{H}$.}
\end{remark}
\begin{definition}
Let $X$ be a two-sided quaternionic Banach space and let $\mathbb{K}\in\{\mathbb{R},\mathbb{H},\mathbb{C}_i:i\in\mathbb{S}\}$. A $\mathbb{K}$-right linear operator on $X$ is a map $T : X \rightarrow X$ such that:
$$T(xq + y) = (Tx)q + Ty \;\;  \text{ for all }\;\; x, y \in X\;\; \text{ and all }\;\; q\in\mathbb{K}.$$
A $\mathbb{K}$-right linear operator $T$ on $X$ is called bounded if
$$\Vert T\Vert:=\sup\{\Vert Tx\Vert: x\in X, \Vert x\Vert=1\}<\infty.$$
\end{definition}
We say that $T$ is a right linear operator if $\mathbb{K}=\mathbb{H}$ and we say that $T$ is a linear operator if $\mathbb{K}=\mathbb{C}_i$ for some $i\in\mathbb{S}$.\\
The set of all right linear bounded operators on $X$ is denoted by $\mathcal{B}_R(X)$. By $\mathcal{B}^{\mathbb{R}}(X)$ we denote  the set of all $\mathbb{R}$-right linear bounded operators on $X$. For $T\in\mathcal{B}^{\mathbb{R}}(X)$, let $\mathcal{N}(T)$ denote the null space of $T$, and let $\mathcal{R}(T)$ denote the range of $T$.\\

In a two-sided quaternionic Banach space $X$, we can define a left and a right quaternionic  multiplication on $\mathcal{B}_R(X)$  by
\begin{align*}
(Tq)x = T(qx)\text{ and } (qT)(x) = q(Tx) \text{ for all } q \in\mathbb{H},\; x\in X  \text{ and all } T \in \mathcal{B}_R(X).
\end{align*}
So $ \mathcal{B}_R(X)$ is a two-sided quaternionic Banach space equipped with the metric $\mathcal{B}_R(X)\times\mathcal{B}_R(X)\ni(A,B)\mapsto\Vert A-B\Vert$.\\

The spectral theory over quaternionic Hilbert spaces has been developed in \cite{5ref7} and \cite{5ref10}.\\

In \cite{5ref7}, the authors extended the definitions of the spectrum and resolvent in quaternionic Banach spaces as follows.
\begin{definition}
\label{5d1}
Let $X$ be a two-sided quaternionic Banach space and let $T \in \mathcal{B}^{\mathbb{R}}(X)$. For $q \in\mathbb{H}$, we set
$$Q_q(T) := T^2 -2\text{Re}(q)T + \vert q\vert^2I,$$
where $I$ is the identity operator on $X$.\\
Let $T\in\mathcal{B}_R(X)$. We define the spherical resolvent set $\rho_S(T)$ of $T$ as
$$\rho_S(T) := \{q \in\mathbb{H} : Q_q(T) \text{ is invertible in } \mathcal{B}_R(X)\},$$
and we define the spherical spectrum $\sigma_S(T)$ of $T$ as  
$$\sigma_S(T):=\mathbb{H}\setminus\rho_S(T).$$
The spherical approximate point spectrum of $T$ is defined by
$$\sigma_{aS}(T):= \{q \in\mathbb{H} : \text{there is a sequence } \{x_n\} \text{ of unit vectors in } X \text{ such that } \Vert Q_q(T)x_n\Vert\rightarrow0\}.$$
The surjectivity spherical spectrum of $T\in \mathcal{B}_R(X)$ is defined by
$$\sigma_{sS}(T):= \{q \in\mathbb{H} : \mathcal{R}(Q_q(T))\neq X\}.$$
\end{definition}
The spherical spectrum $\sigma_S(T)$ decomposes into three disjoint subsets as follows:
\begin{itemize}
\item[(i)] The spherical point spectrum of $T$:
$$\sigma_{pS}(T):= \{q \in\mathbb{H} : \mathcal{N}(Q_q(T))\neq\{0\}\}.$$
\item[(ii)] The spherical residual spectrum of $T$:
$$\sigma_{rS}(T):= \{q \in\mathbb{H} : \mathcal{N}(Q_q(T))=\{0\}\text{ and }\overline{\mathcal{R}(Q_q(T))}\neq X\}.$$
\item[(iii)] The spherical continuous spectrum of $T$:
$$\sigma_{cS}(T):= \{q \in\mathbb{H} : \mathcal{N}(Q_q(T))=\{0\},\;\overline{\mathcal{R}(Q_q(T))}= X\text{ and }\mathcal{R}(Q_q(T))\neq X\}.$$
\end{itemize} 
Let $p$ and $q$ be two quaternions such that $p\in[q]$, then $\text{Re}(p)=\text{Re}(q)$ and $\vert p\vert=\vert q\vert$. Hence  for any $T \in \mathcal{B}_R(X)$, $Q_p(T)=Q_q(T)$. It follows that
\begin{proposition}
\label{5p2}
Let $T\in\mathcal{B}_R(X)$. The sets $\rho_S(T)$, $\sigma_S(T)$, $\sigma_{aS}(T)$, $\sigma_{sS}(T)$, $\sigma_{pS}(T)$, $\sigma_{rS}(T)$ and $\sigma_{cS}(T)$ are axially symmetric.
\end{proposition}
\section{Quaternionic spectrum}
In the following $X$ will be a fixed two-sided quaternionic Banach space. For convenience, we define the mapping:
$$\begin{array}{cccc}
R_q : & X & \rightarrow & X \\ 
 & x & \mapsto & xq.
\end{array} $$
Let $T\in\mathcal{B}^{\mathbb{R}}(X)$ and let $q\in\mathbb{H}$. Define $\Delta_q(T):=R_q-T$. If no confusion can arise, then we write simply $\Delta_q$ instead of $\Delta_q(T)$. Note that $\Delta_q\in\mathcal{B}^{\mathbb{R}}(X)$.
\begin{lemma}
\label{5l1}
Let $T\in\mathcal{B}^{\mathbb{R}}(X)$ and let $q\in\mathbb{H}$. Then $\Delta_q$ is injective (resp. surjective, bijective, has closed range, has dense range) if and only if $\Delta_{\bar{q}}$ is.
\end{lemma}
\begin{proof}
Since $\bar{q}\in[q]$, there exists $j \in \mathbb{S}$ such that $qj=j\bar{q}$. Then for all $x\in X$, $(\Delta_qx)j=\Delta_{\bar{q}}(xj)$, and the result follows immediately.
\end{proof}
\begin{lemma}
\label{5l2}
Let $T\in\mathcal{B}^{\mathbb{R}}(X)$ and let $q\in\mathbb{H}$. Then 
$$Q_q(T)=\Delta_q\Delta_{\bar{q}}=\Delta_{\bar{q}}\Delta_q.$$
\end{lemma}
\begin{proof}
Let $x\in X$. Then
\begin{align}
\label{5e1}
\Delta_q\Delta_{\bar{q}}x=(\Delta_{\bar{q}}x)q-T\Delta_{\bar{q}}x=(x\bar{q}-Tx)q-T(x\bar{q}-Tx)=T^2x -(\bar{q}+q)Tx + q\bar{q}x=Q_q(T)x.
\end{align}
Hence $Q_q(T)=\Delta_q\Delta_{\bar{q}}$.\\
Since $Q_{\bar{q}}(T)=Q_q(T)$, the remaining equality $Q_q(T)=\Delta_{\bar{q}}\Delta_q$ will be proved by interchanging $q$ and $\bar{q}$ in (\ref{5e1}).
\end{proof}
\begin{lemma}
\label{5l3}
Let $T\in\mathcal{B}^{\mathbb{R}}(X)$ and let $q\in\mathbb{H}$. Then $Q_q(T)$ is injective (resp. surjective, bijective, has dense range) if and only if $\Delta_q$ is.
\end{lemma}
\begin{proof}
By Lemmas \ref{5l1} and \ref{5l2}, it is clear that $Q_q(T)$ is injective (resp. surjective, bijective, has dense range) if and only if $\Delta_q$ is.
\end{proof}
The following proposition is an easy consequence of Lemmas \ref{5l2} and \ref{5l3}.
\begin{proposition}
\label{5p1}
Let $T\in\mathcal{B}_R(X)$. Then 
\begin{itemize}
\item[(i)] $\sigma_{S}(T)=\{q\in\mathbb{H}: \Delta_q\text{ is not bijective}\};$
\item[(ii)] $\sigma_{pS}(T)=\{q\in\mathbb{H}: \mathcal{N}(\Delta_q)\neq\{0\}\};$
\item[(iii)] $\sigma_{rS}(T)= \{q \in\mathbb{H} : \mathcal{N}(\Delta_q)=\{0\}\text{ and }\overline{\mathcal{R}(\Delta_q)}\neq X\};$
\item[(iv)] $\sigma_{cS}(T)= \{q \in\mathbb{H} : \mathcal{N}(\Delta_q)=\{0\},\;\overline{\mathcal{R}(\Delta_q)}= X\text{ and }\mathcal{R}(\Delta_q)\neq X\};$
\item[(v)] $\sigma_{aS}(T)= \{q \in\mathbb{H} : \text{there is a sequence } \{x_n\} \text{ of unit vectors in } X \text{ such that } \Vert \Delta_qx_n\Vert\rightarrow0\};$
\item[(vi)] $\sigma_{sS}(T)= \{q \in\mathbb{H} : \mathcal{R}(\Delta_q)\neq X\}.$
\end{itemize} 
\end{proposition}
Let $T\in\mathcal{B}_R(X)$. If $Tx =xq$ for some $q \in\mathbb{H}$ and $x\in X\setminus\{0\}$, then $x$ is called right eigenvector of $T$ with right eigenvalue $q$. The assertion (ii) in the above proposition is well-known (see, for instance, \cite[Proposition 3.1.9]{5ref7} or \cite[Proposition 4.5]{5ref10}).\\

Let $T\in\mathcal{B}_R(X)$ and let $i$ be a fixed element of $\mathbb{S}$. Identify $\mathbb{C}_i$ with $\mathbb{C}$ in the natural way, then $X$ is a complex Banach space with respect to the structure induced by $X$: its sum is the sum of $X$, its complex scalar multiplication is the right scalar multiplication of $X$ restricted to $\mathbb{C}_i$. For short, we denote this space by $X_i$ and $T_i$ the linear operator on $X_i$ defined by $T_i(x):=T(x)$ for all $x\in X$.\\
Note that for all $q\in\mathbb{C}_i$ and all $x\in X$,  $(T_iq)x = T_i(x)q$. We also use the notation $qT_i := T_iq$.
\begin{proposition}
\label{5p5}
Let $T\in\mathcal{B}_R(X)$ and let $q\in\mathbb{C}_i$ then $\Delta_q=I_iq-T_i$. Moreover, if $\Delta_q$ is bijective, then $\Delta_q^{-1}$ is the resolvent of $T_i$ at $q$.
\end{proposition}
\begin{proof}
Let $x\in X$, then $(I_iq-T_i)x=(I_iq)x-T_ix=xq-Tx=\Delta_qx$. Hence $\Delta_q=I_iq-T_i$. If $\Delta_q$ is bijective, then so is $I_iq-T_i$. Thus $\Delta_q^{-1}=(I_iq-T_i)^{-1}$ is the resolvent of $T_i$ at $q$.
\end{proof}
\begin{theorem}
\label{5t1}
Let $T\in\mathcal{B}_R(X)$ and let $q\in\mathbb{H}$ with $\limsup_{n\rightarrow+\infty}\Vert T^n\Vert^{\frac{1}{n}} < \vert q\vert$. Then
$$Q_q(T)^{-1}=\sum_{n=0}^{+\infty}\big(\sum_{k=0}^{n}q^{-k-1}\bar{q}^{\;-n+k-1}\big)T^n,$$
where this series converges in the operator norm.
\end{theorem}
\begin{proof}
Note that for all $p\in[q]$, $Q_p(T)=Q_q(T)$. Then we may assume that $q\in\mathbb{C}_i$. Since 
$$\limsup_{n\rightarrow+\infty}\Vert T_i^n\Vert^{\frac{1}{n}}=\limsup_{n\rightarrow+\infty}\Vert T^n\Vert^{\frac{1}{n}} < \vert q\vert,$$
$I_iq-T_i=\Delta_q$ and $I_i\bar{q}-T_i=\Delta_{\bar{q}}$ are invertible in $\mathcal{B}(X_i)$. Moreover,
\begin{align}
\Delta_q^{-1}=(I_iq-T_i)^{-1}=\sum_{n=0}^{+\infty}T_i^nq^{-n-1}\text{ and }\Delta_{\bar{q}}^{-1}=(I_i\bar{q}-T_i)^{-1}=\sum_{n=0}^{+\infty}T_i^n\bar{q}^{-n-1}.\label{5eq1}
\end{align}
Then  
$$Q_q(T)^{-1}=\Delta_q^{-1}\Delta_{\bar{q}}^{-1}=\sum_{n=0}^{+\infty}T_i^n\sum_{k=0}^{n}q^{-k-1}{\bar{q}}^{\;-n+k-1}.$$
Since $\sum_{k=0}^{n}q^{-k-1}{\bar{q}}^{\;-n+k-1}=\overline{\sum_{k=0}^{n}q^{-k-1}{\bar{q}}^{\;-n+k-1}}$, $\sum_{k=0}^{n}q^{-k-1}{\bar{q}}^{\;-n+k-1}\in\mathbb{R}$. Hence
$$T_i^n\sum_{k=0}^{n}q^{-k-1}{\bar{q}}^{\;-n+k-1}=T^n\sum_{k=0}^{n}q^{-k-1}{\bar{q}}^{\;-n+k-1}\;\;\text{ for all }\; n\in\mathbb{N}.$$
Thus 
$$Q_q(T)^{-1}=\sum_{n=0}^{+\infty}\big(\sum_{k=0}^{n}q^{-k-1}\bar{q}^{\;-n+k-1}\big)T^n,$$
this series converges in the operator norm, because so are the series in (\ref{5eq1}). This completes the proof.
\end{proof}
\begin{remark}
\emph{Let $q\in\mathbb{H}$ and let $a_n:=\sum_{k=0}^{n}q^{-k-1}{\bar{q}}^{\;-n+k-1}$. Since $a_n=\overline{a_n}$, we can write 
$$Q_q(T)^{-1}=\sum_{n=0}^{+\infty}T^n\big(\sum_{k=0}^{n}q^{-k-1}\bar{q}^{\;-n+k-1}\big).$$}
\end{remark}
\begin{proposition}
\label{5p3}
Let $T\in\mathcal{B}_R(X)$. Then 
$$\sigma_S(T)=[\sigma(T_i)].$$
\end{proposition}
\begin{proof}
By Proposition \ref{5p1}, $\sigma(T_i)\subseteq\sigma_S(T)$. Since $\sigma_S(T)$ is axially symmetric, $[\sigma(T_i)]\subseteq\sigma_S(T)$.\\
Conversely, let $q\in\sigma_S(T)$. Then there exists $s\in\mathbb{S}$ such that $sq\bar{s}\in\mathbb{C}_i\cap\sigma_S(T)$, then by Proposition \ref{5p5}, $sq\bar{s}\in\sigma(T_i)$. Thus $q\in[\sigma(T_i)]$.
\end{proof}
\begin{theorem}[Compactness of the S-spectrum]
Let $T\in\mathcal{B}_R(X)$. The S-spectrum $\sigma_S(T)$ of $T$ is a nonempty compact set contained in the closed ball $\{q\in\mathbb{H}:\vert q\vert\leq\Vert T\Vert\}$.
\end{theorem}
\begin{proof}
It is well-known that $\sigma(T_i)$ is a nonempty compact subset of $\mathbb{C}_i$. Hence $[\sigma(T_i)]$ is compact in $\mathbb{H}$. Then by Proposition \ref{5p3}, $\sigma_S(T)$ is a nonempty compact set. That $\sigma_S(T)$ contained in the closed ball $\{q\in\mathbb{H}:\vert q\vert\leq\Vert T\Vert\}$ follows from Theorem \ref{5t1}. 
\end{proof}
Let $T \in \mathcal{B}_R(X)$. Then the S-spectral radius of $T$ is defined to be the nonnegative real number
$$r_S(T) := \sup\{\vert q\vert : q \in \sigma_S(T)\}.$$
\begin{proposition}
Let $T\in\mathcal{B}_R(X)$. Then 
$$r_S(T)=r(T_i).$$
In particular, $$r_S(T)=\lim_{n\rightarrow\infty}\Vert T^n\Vert^{\frac{1}{n}}.$$
\end{proposition}
\begin{proof}
By Proposition \ref{5p3}, we have $\sigma_S(T)=[\sigma(T_i)]$. Clearly if $p,q\in\mathbb{H}$ are conjugate, then $\vert p\vert=\vert q\vert$. Hence $\{\vert q\vert : q \in \sigma_S(T)\}=\{\vert q\vert : q \in \sigma(T_i)\}$. Thus 
$$r_S(T) := \sup\{\vert q\vert : q \in \sigma(T_i)\}.$$
And so $$r_S(T)=\lim_{n\rightarrow\infty}\Vert T_i^n\Vert^{\frac{1}{n}}=\lim_{n\rightarrow\infty}\Vert T^n\Vert^{\frac{1}{n}}.$$
\end{proof}
Let $P(q) =\sum_{i=0}^{n}\alpha_iq^i$ be a polynomial with coefficients $\alpha_i\in\mathbb{R}$. For $T\in\mathcal{B}_R(X)$ we set $P(T) =\sum_{i=0}^{n}\alpha_iT^i$. The S-spectrum of $T$ and $P(T)$ and related in the following way:
\begin{theorem}[The spectral mapping theorem for polynomials]
\label{5t10}
Let $T \in\mathcal{B}_R(X)$ and let $P$ be a polynomial with real coefficients. Then
$$\sigma_S(P(T))=P(\sigma_S(T)):=\{P(q):q\in\sigma_S(T)\}.$$
\end{theorem}
\begin{proof}
We have $\sigma(P(T_i))=P(\sigma(T_i))$. Then by Proposition \ref{5p3}, $\sigma_S(P(T))=[P(\sigma(T_i))]=P([\sigma(T_i)])=P(\sigma_S(T)).$
\end{proof}
\begin{remark}
\emph{Note that Theorem \ref{5t10} would be fail if the polynomial coefficients are not real. To see this, consider the polynomial $P(q) =\mathrm{i}q$ and $I$ the identity operator on $X$. Then $\sigma_S(P(I))=\sigma_S(\mathrm{i} I)=\mathbb{S}$ and $P(\sigma_S(I))=\{\mathrm{i}\}$.}
\end{remark}
\section{S-functional calculus}
In this section we will see that the S-functional calculus can be obtained by the Riesz-Dunford functional calculus for complex linear operators.
\begin{definition}[{\cite[Definition 2.1.30]{5ref7}}, Slice Cauchy domain]
An axially symmetric open set $D \subset\mathbb{H}$ is called a slice Cauchy domain if $D \cap\mathbb{C}_i$ is a Cauchy domain in $\mathbb{C}_i$ for every $i\in\mathbb{S}$. More precisely, $D$ is a slice Cauchy domain if for every $i\in\mathbb{S}$ the boundary $\partial(D\cap\mathbb{C}_i)$ of $D\cap\mathbb{C}_i$ is the union of a finite number of nonintersecting piecewise continuously differentiable Jordan curves in $\mathbb{C}_i$.
\end{definition}

Let $T\in\mathcal{B}_R(X)$ and let $D$ be a bounded slice Cauchy domain that contains $\sigma_S(T)$. Let $f$ be an intrinsic slice hyperholomorphic function with $\overline{D}\subset\mathcal{D}(f)$, let $f_i:=f\vert_{\mathcal{D}(f)\cap\mathbb{C}_i}$ and $D_i:=D\cap\mathbb{C}_i$ ($D_i$ is a bounded slice Cauchy domain in $\mathbb{C}_i$). Since $\sigma(T_i)=\sigma_S(T)\cap\mathbb{C}_i$ and $\Delta_s^{-1}$ \big($s\in\rho_S(T)\cap\mathbb{C}_i$\big) is the resolvent of $T_i$,
\begin{align}
f_i(T_i)=\frac{1}{2\pi i}\int_{\partial D_i}\Delta_s^{-1}f_i(s)ds. \label{5e4}
\end{align}
\begin{definition}[{\cite[Definition 3.1.10]{5ref7}}]
Let $T\in\mathcal{B}_R(X)$. For $s\in\sigma_S(T)$, we define the left S-resolvent operator as
$$S_L^{-1}(s,T):=-Q_s(T)^{-1}(T-\bar{s}I),$$
and the right S-resolvent operator as
$$S_R^{-1}(s,T):=-(T-\bar{s}I)Q_s(T)^{-1}.$$
\end{definition}
\begin{theorem}[{\cite[Theorem 3.1.6]{5ref7}}]
Let $T\in\mathcal{B}_R(X)$ and let $s\in\mathbb{H}$ with $\Vert T\Vert<\vert s\vert$. Then 
$$S_L^{-1}(s,T)=\sum_{n=0}^{+\infty}T^ns^{-n-1};$$
$$S_R^{-1}(s,T)=\sum_{n=0}^{+\infty}s^{-n-1}T^n.$$
These series converge in the operator norm.
\end{theorem}
In \cite{5ref7} the authors defined the S-functional calculus as follows. For every right (resp. left) slice hyperholomorphic function $f$ and every bounded slice Cauchy domain $D$ such that  $\sigma_S(T)\subset D\subset\overline{D}\subset\mathcal{D}(f)$, 
$$f(T):=\frac{1}{2\pi}\int_{\partial D_i}f(s)ds_iS_R^{-1}(s,T)\;\;\;\;\left(\text{resp. }f(T):=\frac{1}{2\pi}\int_{\partial D_i}S_L^{-1}(s,T)ds_if(s)\right).$$
\begin{proposition}
\label{5p4}
Let $T\in\mathcal{B}_R(X)$ and let $D$ be a bounded slice Cauchy domain such that $\sigma_S(T)\subset D$. Let $f$ be an intrinsic slice hyperholomorphic function such that $\overline{D}\subset\mathcal{D}(f)$. Then
\begin{align}
\frac{1}{2\pi}\int_{\partial D_i}S_L^{-1}(s,T)ds_if(s)=\frac{1}{2\pi i}\int_{\partial D_i}\Delta_s^{-1}f(s)ds. \label{5e5}
\end{align}
That is $f(T)=f_i(T_i)$.
\end{proposition}
\begin{proof}
The two integrals in (\ref{5e5}) are independent of the choice of the bounded slice Cauchy domain $D$. Let $D$ be the ball $B_r(0)$ with $\Vert T\Vert < r$. Since $\Delta_s^{-1}=\sum_{n=0}^{+\infty}T_i^ns^{-n-1}$ converges uniformly on $\partial D_i$,
$$\frac{1}{2\pi i}\int_{\partial D_i}\Delta_s^{-1}f(s)ds=\sum_{n=0}^{+\infty}T_i^n\frac{1}{2\pi i}\int_{\partial D_i}s^{-n-1}f(s)ds.$$
And $S_L^{-1}(s,T)=\sum_{n=0}^{+\infty}T^ns^{-n-1}$ converges uniformly on $\partial D_i$,
\begin{align*}
\frac{1}{2\pi}\int_{\partial D_i}S_L^{-1}(s,T)ds_if(s)&=\sum_{n=0}^{+\infty}T^n\frac{1}{2\pi}\int_{\partial D_i}s^{-n-1}ds_if(s)\\
&=\sum_{n=0}^{+\infty}T^n\frac{1}{2\pi i}\int_{\partial D_i}s^{-n-1}f(s)ds.
\end{align*}
Since $f$ is intrinsic, by Lemma \ref{5l6} (iii),
$$\overline{\frac{1}{2\pi i}\int_{\partial D_i}s^{-n-1}f(s)ds}=\frac{1}{2\pi i}\int_{\partial D_i}s^{-n-1}f(s)ds.$$
Hence $\frac{1}{2\pi i}\int_{\partial D_i}s^{-n-1}f(s)ds\in\mathbb{R}$, and so  
$$T^n\frac{1}{2\pi i}\int_{\partial D_i}s^{-n-1}f(s)ds=T_i^n\frac{1}{2\pi i}\int_{\partial D_i}s^{-n-1}f(s)ds\;\;\text{ for all }\;\;n\in\mathbb{N}.$$
Thus $$\frac{1}{2\pi}\int_{\partial D_i}S_L^{-1}(s,T)ds_if(s)=\frac{1}{2\pi i}\int_{\partial D_i}f(s)\Delta_s^{-1}ds.$$
\end{proof}
\begin{remark}
\emph{It is easy to see that $f(T)_i=f_i(T_i)$.}
\end{remark}
\begin{lemma} 
\label{5l4}
Let $U\subseteq\mathbb{H}$ be an axially symmetric open set and let $f$ be a right (resp. left) slice hyperholomorphic function on $U$. Then there exist intrinsic slice hyperholomorphic functions $f_1,f_2,f_3$ and $f_4$ on $U$ such that 
$$f=f_1+\mathrm{i} f_2+\mathrm{j}f_3+\mathrm{k}f_4\;\;\;\; \left(\text{resp. }f=f_1+ f_2\mathrm{i}+f_3\mathrm{j}+f_4\mathrm{k}\right).$$
\end{lemma}
\begin{proof}
Let $j\in\mathbb{S}$ and $\alpha, \beta\in\mathbb{R}$ such that $\alpha+\beta j\in U$. Write
$f(\alpha+\beta j) = f_0(\alpha, \beta) + f_1(\alpha, \beta)j$. Let $f_{n,m}$ where $n=0,1$ and $m=1,2,3,4$ be real functions such that
$$f_0(\alpha, \beta)=f_{0,1}(\alpha, \beta)+f_{0,2}(\alpha, \beta)\mathrm{i}+f_{0,3}(\alpha, \beta)\mathrm{j}+f_{0,4}(\alpha, \beta)\mathrm{k};$$
$$f_1(\alpha, \beta)=f_{1,1}(\alpha, \beta)+f_{1,2}(\alpha, \beta)\mathrm{i}+f_{1,3}(\alpha, \beta)\mathrm{j}+f_{1,4}(\alpha, \beta)\mathrm{k}.$$
Let 
$$f_m(\alpha+\beta j):=f_{0,m}(\alpha, \beta)+f_{1,m}(\alpha, \beta)j\;\;\;\;\;\;\alpha+\beta j\in U,\;\;m=1,2,3,4.$$
Since $\{1,\mathrm{i},\mathrm{j},\mathrm{k}\}$ is a basis of $\mathbb{H}$, $f_1,f_2,f_3$ and $f_4$ are intrinsic slice hyperholomorphic functions on $U$.\\
The left slice hyperholomorphic case can be proved similarly.
\end{proof}
Let $T\in\mathcal{B}_R(X)$ and let $D$ be a bounded slice Cauchy domain such that $\sigma_S(T)\subset D$. Let $f$ be a right slice hyperholomorphic function such that $\overline{D}\subset\mathcal{D}(f)$. By Lemma \ref{5l4}, there exist intrinsic slice hyperholomorphic functions $f_1,f_2,f_3,f_4$ such that $f=f_1+\mathrm{i} f_2+\mathrm{j}f_3+\mathrm{k}f_4$. Then
\begin{align*}
f(T)&=\frac{1}{2\pi}\int_{\partial D_i}f(s)ds_iS_R^{-1}(s,T)\\
&=\frac{1}{2\pi}\int_{\partial D_i}f_1(s)ds_iS_R^{-1}(s,T)+\mathrm{i}\frac{1}{2\pi}\int_{\partial D_i}f_2(s)ds_iS_R^{-1}(s,T)\\
&+\mathrm{j}\frac{1}{2\pi}\int_{\partial D_i}f_3(s)ds_iS_R^{-1}(s,T)+\mathrm{k}\frac{1}{2\pi}\int_{\partial D_i}f_4(s)ds_iS_R^{-1}(s,T).
\end{align*}
By Proposition \ref{5p4}, for all $m=1,2,3,4$, the integral $\frac{1}{2\pi}\int_{\partial D_i}f_m(s)ds_iS_R^{-1}(s,T)$ can be written as $\frac{1}{2\pi i}\int_{\partial D_i}\Delta_s^{-1}f_m(s)ds$. Hence
\begin{align*}
f(T)&=\frac{1}{2\pi i}\int_{\partial D_i}\Delta_s^{-1}f_1(s)ds+\mathrm{i}\frac{1}{2\pi i}\int_{\partial D_i}\Delta_s^{-1}f_2(s)ds\\
&+\mathrm{j}\frac{1}{2\pi i}\int_{\partial D_i}\Delta_s^{-1}f_3(s)ds+\mathrm{k}\frac{1}{2\pi i}\int_{\partial D_i}\Delta_s^{-1}f_4(s)ds.
\end{align*}
Similarly, for any left slice hyperholomorphic function such that $\overline{U}\subset\mathcal{D}(f)$, there exist intrinsic slice hyperholomorphic functions $f_1,f_2,f_3,f_4$ such that
\begin{align*}
f(T)&=\frac{1}{2\pi i}\int_{\partial D_i}\Delta_s^{-1}f_1(s)ds+\left(\frac{1}{2\pi i}\int_{\partial D_i}\Delta_s^{-1}f_2(s)ds\right)\mathrm{i}\\
&+\left(\frac{1}{2\pi i}\int_{\partial D_i}\Delta_s^{-1}f_3(s)ds\right)\mathrm{j}+\left(\frac{1}{2\pi i}\int_{\partial D_i}\Delta_s^{-1}f_4(s)ds\right)\mathrm{k}.
\end{align*}
Thus the S-functional calculus is obtained easily by the Riesz-Dunford functional calculus for complex bounded linear operators.\\

Now we are able to show that most of the properties that hold for the Riesz-Dunford functional calculus can be extended to the S-functional calculus. In \cite{5ref7} the authors extended these properties with several additional efforts.\\

Let $T \in\mathcal{B}_R(X)$ and let $K$ be a compact axially symmetric subset of $\mathbb{H}$. We denote by $\mathcal{SH}_L(K)$, $\mathcal{SH}_R(K)$ and $\mathcal{N}(K)$ the sets of all left, right and intrinsic slice hyperholomorphic functions $f$, where the domain of the function $f$ contains the compact axially symmetric set $K$.
\begin{theorem}[{\cite[Theorem 4.1.3]{5ref7}}, Product rule]
\label{5t7}
Let $T \in\mathcal{B}_R(X)$ and let $f\in\mathcal{N}\big(\sigma_S(T)\big)$ and $g\in\mathcal{SH}_L\big(\sigma_S(T)\big)$ or let $f\in\mathcal{SH}_R\big(\sigma_S(T)\big)$ and $g\in\mathcal{N}\big(\sigma_S(T)\big)$. Then
$$\big(fg\big)\big(T\big) = f(T)g(T).$$
\end{theorem}
\begin{proof}
Let $f\in\mathcal{SH}_R\big(\sigma_S(T)\big)$ and $g\in\mathcal{N}\big(\sigma_S(T)\big)$. By Lemma \ref{5l4} there exist intrinsic slice hyperholomorphic functions $f_1,f_2,f_3$ and $f_4$ such that 
$$f=f_1+\mathrm{i} f_2+\mathrm{j}f_3+\mathrm{k}f_4.$$
Then $fg=f_1g+\mathrm{i} f_2g+\mathrm{j}f_3g+\mathrm{k}f_4g$. Hence 
$$(fg)(T)=(f_1g)(T)+\mathrm{i} (f_2g)(T)+\mathrm{j}(f_3g)(T)+\mathrm{k}(f_4g)(T).$$
Note that $(f_mg)(T)=(f_mg)_i(T_i)=(f_m)_i(T_i)g_i(T_i)=f_m(T)g(T)$ $\;(m=1,2,3,4)$. Hence $(fg)(T)=f(T)g(T).$
\end{proof}
\begin{theorem}[{\cite[Theorem 4.2.1]{5ref7}}, The spectral mapping theorem]
\label{5t8}
Let $T \in\mathcal{B}_R(X)$ and let $f\in\mathcal{N}\big(\sigma_S(T)\big)$. Then
$$\sigma_S\big(f(T)\big)=f\big(\sigma_S(T)\big):=\{f(q):q\in\sigma_S(T)\}.$$
\end{theorem}
\begin{proof}
Note that $\sigma\big(f(T)_i\big)=\sigma\big(f_i(T_i)\big)=f_i\big(\sigma(T_i)\big)$. Thus by Proposition \ref{5p3},
$$\sigma_S\big(f(T)\big)=\big[\sigma\big(f(T)_i\big)\big]=\big[f_i\big(\sigma(T_i)\big)\big]=f\big([\sigma(T_i)]\big)=f\big(\sigma_S(T)\big).$$
\end{proof}
\begin{theorem}[{\cite[Theorem 4.2.4]{5ref7}}, Composition rule]
\label{5t9}
Let $T \in\mathcal{B}_R(X)$ and let $f\in\mathcal{N}\big(\sigma_S(T)\big)$. If $g\in\mathcal{SH}_L\big(\sigma_S\big(f(T)\big)\big)$, then $g\circ f\in\mathcal{SH}_L\big(\sigma_S(T)\big)$, and if $g\in\mathcal{SH}_R\big(f\big(\sigma_S(T)\big)\big)$, then $g\circ f\in\mathcal{SH}_R\big(\sigma_S(T)\big)$. In both cases,
$$g\big(f(T)\big) = \big(g \circ f\big)\big(T\big).$$
\end{theorem}
\begin{proof}
Let $f\in\mathcal{N}\big(\sigma_S(T)\big)$ and $g\in\mathcal{SH}_R\big(\sigma_S(T)\big)$. By lemma \ref{5l4} there exist intrinsic slice hyperholomorphic functions $g_1,g_2,g_3$ and $g_4$ such that 
$$g=g_1+\mathrm{i} g_2+\mathrm{j}g_3+\mathrm{k}g_4.$$
Then $g\circ f=g_1\circ f+\mathrm{i} g_2\circ f+\mathrm{j}g_3\circ f+\mathrm{k}g_4\circ f$. Hence 
$$\big(g\circ f\big)\big(T\big)=\big(g_1\circ f\big)\big(T\big)+\mathrm{i} \big(g_2\circ f\big)\big(T\big)+\mathrm{j}\big(g_3\circ f\big)\big(T\big)+\mathrm{k}\big(g_4\circ f\big)\big(T\big).$$
We have $\big(g_m\circ f\big)\big(T\big)=\big(g_m\circ f\big)_i\big(T_i\big)=\big(g_m\big)_i\big(f_i(T_i)\big)=g_m\big(f(T)\big)$. Hence $\big(g\circ f\big)\big(T\big)=g\big(f(T)\big).$ The remaining case can be proved similarly.
\end{proof}
\begin{theorem}
Let $T \in\mathcal{B}_R(X)$ and let $\alpha\in\mathbb{R}$ such that $\alpha\notin \sigma_S(T)$. Then we have 
$$\mathrm{dist}\big(\alpha,\sigma_S(T)\big)=\frac{1}{r_S\big((\alpha I-T)^{-1}\big)}.$$
\end{theorem}
\begin{proof}
Let $U\subseteq\mathbb{H}$ be an axially symmetric open set containing $\sigma_S(T)$, but not $\alpha$. Then $f(q)=(\alpha -q)^{-1}$ is an intrinsic slice hyperholomorphic function on $U$. So by Theorem \ref{5t8} we have 
$$\sigma_S\big((\alpha I-T)^{-1}\big)=\{(\alpha -q)^{-1}:q\in\sigma_S(T)\}.$$
So in particular, 
$$r_S\big((\alpha I-T)^{-1}\big)=\sup\{\vert(\alpha -q)^{-1}\vert:q\in\sigma_S(T)\}=1/\inf\{\vert\alpha -q\vert:q\in\sigma_S(T)\}=1/\mathrm{dist}\big(\alpha,\sigma_S(T)\big).$$
\end{proof}
Let $T \in\mathcal{B}_R(X)$, we define
$$\exp(T)=\sum_{n=0}^{\infty}\frac{1}{n!}T^n$$
this series converges in the operator norm. 
\begin{theorem}
Let $T \in\mathcal{B}_R(X)$. Suppose that $0$ is in the unbounded component of $\rho_S(T)$. Then there exists $S \in\mathcal{B}_R(X)$ such that $T=\exp(S)$. In particular, for every integer $n \geq 1$ there exists $S \in\mathcal{B}_R(X)$ such that $S^n = T$. 
\end{theorem}
\begin{proof}
Since $0$ is in the unbounded component of $\rho_S(T)$, $0$ is in the unbounded component of $\rho(T_i)$. By Lemma \ref{5l1}, $\sigma(T_i)$ is symmetric with respect to the real axis. Hence there exists a symmetric simply connected open set $O$ of $\mathbb{C}_i$ containing $\sigma(T_i)$ but not $0$.  By \cite[Theorem 13.18]{5ref11} there exists $f$ holomorphic on $O$ such that $s=\exp(f(s))$ for all $s\in O$. The function $f$ can be extended to a left slice hyperholomorphic function $\tilde{f}$ on $[O]$ (see, \cite[Lemma 2.1.11]{5ref7}). Thus by \cite[Theorem 2.1.8]{5ref7} $s=\exp(\tilde{f}(s))$ for all $s\in [O]$, and so $T=\exp(\tilde{f}(T))$, and $T=\exp(\frac{1}{n} \tilde{f}(T))^n$ for all $n \geq 1$.
\end{proof}

\end{document}